\documentclass[letterpaper,11pt,reqno]{amsart}
\RequirePackage[utf8]{inputenc}
\usepackage[portrait,margin=2.5cm]{geometry}
\usepackage{amsmath,mathtools}
\usepackage{bbm}
\usepackage[dvipsnames]{xcolor}
\usepackage{latexsym, amssymb, amsmath, amscd, amsthm, amsxtra}
\usepackage[numeric,initials,nobysame]{amsrefs}
\usepackage{mathtools}
\usepackage{enumerate}
\usepackage[all]{xy}
\usepackage{mathrsfs}
\usepackage{fancyhdr}
\usepackage{listings}
\usepackage{hyperref}

\usepackage{algorithm}
\usepackage{algpseudocode}
\usepackage{framed}
\usepackage{lipsum}

\makeatletter
\newenvironment{breakablealgorithm}
{
		\begin{center}
			\refstepcounter{algorithm}
			\hrule height.8pt depth0pt \kern2pt
			\renewcommand{\caption}[2][\relax]{
				{\raggedright\textbf{\ALG@name~\thealgorithm} ##2\par}%
				\ifx\relax##1\relax 
				\addcontentsline{loa}{algorithm}{\protect\numberline{\thealgorithm}##2}%
				\else 
				\addcontentsline{loa}{algorithm}{\protect\numberline{\thealgorithm}##1}%
				\fi
				\kern2pt\hrule\kern2pt
			}
		}{
		\kern2pt\hrule\relax
	\end{center}
}
\makeatother

\hypersetup{
	colorlinks   = true, 
	urlcolor     = blue, 
	linkcolor    = blue, 
	citecolor   = red 
}

\newtheorem{theorem}{Theorem}[section]
\newtheorem{lemma}[theorem]{Lemma}

\newtheorem{proposition}[theorem]{Proposition}

\theoremstyle{definition}
\newtheorem{definition}[theorem]{Definition}

\newcommand{\OPT}{{\sf OPT} }

\newcommand{\ER}{Erd\H{o}s-R\'enyi\ }

\newcommand{\R}{\mathbb{R}}

\DeclareMathOperator{\avg}{Avg}

\newcommand{\fD}{\mathfrak{D}}\newcommand{\fF}{\mathfrak{F}}

\usepackage{tikz}
\usetikzlibrary{patterns}

\usepackage{enumerate}

\newenvironment{enumeratea}{\begin{enumerate}[\upshape (a)]}{\end{enumerate}}

\title{Finding a dense submatrix of a random matrix. Sharp bounds for  online algorithms}

\author{%
  Shankar Bhamidi$^*$,\,%
  David Gamarnik$^\dagger$,\,%
  Shuyang Gong$^\ddagger$%
}

\thanks{\textsuperscript{$*$}Department of Statistics and Operations Research, University of North Carolina at Chapel Hill; Email: \texttt{bhamidi@email.unc.edu}}
\thanks{\textsuperscript{$\dagger$}MIT Sloan School of Management; Email: \texttt{gamarnik@mit.edu}}
\thanks{\textsuperscript{$\ddagger$}School of Mathematical Sciences, Peking University; Email: \texttt{gongshuyang@stu.pku.edu.cn}}

\date{\today}
\makeatletter
  \renewcommand{\and}{, }
\makeatother

\begin{document}

\maketitle

\begin{abstract}    
We consider the problem of finding a dense submatrix of a matrix with \emph{i.i.d.} Gaussian entries, where 
density is measured by average value. This problem arose from practical applications in biology and social sciences \cites{madeira-survey,shabalin2009finding} and is known to exhibit a computation-to-optimization gap between the optimal value and best values achievable by existing polynomial time algorithms. In this paper we consider the class of online algorithms, which includes the  best known algorithm for this problem, and derive a tight approximation factor ${4\over 3\sqrt{2}}$ for this class. The result is established using a simple implementation of recently developed Branching-Overlap-Gap-Property \cite{huang2025tight}. We further
extend our results to $(\mathbb R^n)^{\otimes p}$ tensors with \emph{i.i.d.} Gaussian  entries, for which 
the approximation factor is proven to be ${2\sqrt{p}/(1+p)}$.

 \end{abstract}

\section{Introduction}
\label{sec-intro}

We consider the algorithmic problem of finding a densest submatrix of 
fixed size of a given matrix. Specifically, given a matrix 
$G\in \R^{n\times n}$ and  $k$, the goal is to find a set
of rows and columns combination
$I,J\subset [n], |I|=|J|=k$ which maximize
average entry of the submatrix indexed by $I$ and $J$, namely
\begin{align}
\max_{I,J\subset [n], |I|=|J|=k}
{1\over k^2}\sum_{i\in I, j\in J} G_{ij}.
\label{eq:max-ave}
\end{align}
Similarly, we address the case when $G\in \R^{n\otimes p}$ is a $n^{\otimes p}$ tensor
with side-length $n$ and rank $p$. Anologous to \eqref{eq:max-ave}, the goal is finding
a collection of sets $I_1,\ldots,I_p\subset [n]$ of size $k$
which maximize
\begin{align}
\max_{I_1,\ldots,I_p\subset [n]: |I_1|=\cdots=|I_p|=k}
{1\over k^p}\sum_{i_1\in I_1, \ldots, i_p\in I_p} 
G_{i_1,\ldots,i_p}. \label{eq:tensor}
\end{align}
For both problems,  we consider the setting when $G$ is 
random with i.i.d. standard normal entries.

This question arose recently in context of 
several applications including genetics and bioinformatics, text-mining and reccomender systems \cites{madeira-survey,shabalin2009finding,pontes2015biclustering} and prompted a series
of probabilistic algorithmic studies which we now
describe. The first paper which conducted 
a theoretical analysis of this question is~\cite{bhamidi2017energy} which deals with
 the   matrix case $p=2$. Among many other
results, the authors showed that as long as $k$ grows smaller than any power of $n$, $k=\exp(o(\log n))$, 
 the optimal value (\ref{eq:max-ave}) converges to $\eta_{\rm OPT}\triangleq 
2\sqrt{\log n/k}$ whp as $n$
increases. This answer is rather intuitive. It is 
the extremal value of ${n\choose k}^2$ many (the
total number of $k\times k$ submatrices) independent
Gaussians with variance $1/k^2$ (the variance
for every fixed submatrix). 

Next, \cite{bhamidi2017energy} considered a natural
so-called $\mathcal{LAS}$ 
(\emph{Large Average Submatrix}) algorithm,  where, starting from an initial submatrix, one alternates between keeping the rows fixed and finding the $k$ columns with the largest average amongst the fixed set of rows and then keeping the columns fixed and finding the $k$ rows with the largest average and iterating this procedure till reaching a fixed point of this operation.    A typical fixed point of this algorithm is one where there is no incentive to move to a different set of $k$ columns, keeping the rows fixed and the same for rows, keeping the columns fixed. The paper \cite{bhamidi2017energy} analyzed the structure of a sub-matrix satisfying these constraints for fixed size $k$ and showed that the average of such a fixed point was asymptotically $\sqrt{2}\sqrt{\log n/k}$, 
namely a factor $1/\sqrt{2}$ smaller than the optimum. Based on this  result it was conjectured that the $\mathcal{LAS}$ algorithm starting from a random initial submatrix would converge to a local optima with the above scaling of its average.

It was further suggested that the model might
exhibit an infamous computation-to-optimization
(also commonly called statistics-to-computation) gap:
known computationally feasible algorithms under-perform by at least 
a constant factor with respect to optimality. 
The  performance of the $\mathcal{LAS}$ algorithm conjectured in~\cite{bhamidi2017energy}  was established
in ~\cite{GL18AOS}. In the same paper a different iterative procedure 
called \emph{Incremental Greedy Procedure} ($\mathcal{IGP}$)
was proposed which whp produces a matrix with an improved average value 
\begin{align}
\eta_{\rm ALG}\triangleq 
    {4\over 3}\sqrt{2\log n/k}={2\sqrt{2}\over 3}{\OPT} \approx 0.942..~{\OPT}. \label{eq:eta-ALG}
\end{align}
This value, till date,  remains the best known algorithmic performance within the class of polynomial time algorithms. 
In the other direction,  an Overlap Gap Property (OGP) was established
in \cite{GL18AOS} for submatrices with values larger 
than $\eta_{\rm 2-OGP} \OPT$, where $\eta_{\rm 2-OGP} ={5\over 3\sqrt{3}} = 0.962..~$. Roughly speaking, this says
that whp,  pairs of $k\times k$ submatrices with average
values  above $\eta_{\rm 2-OGP}$ are either fairly close to each other 
(have the  number of common rows and columns at least some value $\nu_2$) or fairly far from
each other (the  number of common rows and columns at at most some value $\nu_1<\nu_2$).
OGP is a known barrier to classes of 
algorithms~\cites{gamarnik2021overlap,gamarnik2025turing,gamarnik2022disordered}. We use ${\rm 2-OGP}$ to indicate the matrix versus the more general tensor setting that we describe next. 

In this paper using OGP type argument
we establish that $\eta_{\rm ALG}$ given
in (\ref{eq:eta-ALG}) is optimal within 
the class of online algorithms appropriately
defined. Importantly, the $\mathcal{IGP}$ falls into
the category of online algorithm, and thus our result
establishes sharp performance bounds for online algorithms for this problem. We extend our result
to the case of random tensors (\ref{eq:tensor}). 
The optimal value for densest $k$-tensors
is found to be $\sqrt{2p \log n\over k^{p-1}}$ 
as was recently verified in
\cite{erba2025maximumaverage}
and \cite{abhishek2025large}.

We derive 
the tight approximation factor
for online algorithms and show that is  ${2\sqrt{p}\over p+1}$.
We do conjecture though 
that no polynomial time algorithm can improve upon
this value when $k$ is a growing function of $n$ (when $k$ is constant, the brute-force algorithm
itself is polynomial time). Naturally this remains 
beyond the 
scope of currently approachable results, since this
result would subsume proving $P\ne NP$.

The OGP construction adopted in this paper is inspired and follows 
closely the one used in~\cite{du2025algorithmic},
which considered the problem of graph alignment problem of two independent copies of an \ER 
random graph. As in \cite{du2025algorithmic}, 
we consider a certain continuous ultrametric tree structure of solutions with appropriately
high objective value. More precisely,  similar to
\cite{du2025algorithmic},  we show that there \emph{does
not} exist a regular tree of subtensors of 
our random tensor such that (a) the nodes of the trees correspond to subtensor, and 
the leaves of the tree
correspond to subtensors with average value
larger than $\eta_{\rm ALG}$, (b) distance between
neighbors in a tree, measured by Hamming distance
is appropriately small, and (c) the depth and the
degree of the tree is appropriately large. This non-existence property is what we
call branching or ultrametric-OGP.
At the same time we show that if a putative
online algorithm for finding matrix with value
larger than $\eta_{\rm ALG}$ exists, one can 
build a tree of such algorithm-produced subtensors
satisfying conditions (a),(b) and (c), thus 
leading to a contradiction. This is the  outline of our proof approach for our main result.

The ultrametric construction
of OGP described above was pioneered in ~\cite{huang2025tight} and
coined Branching-OGP in this paper. Ultrametricity arises naturally from the study of spin glasses and is an essential
tool for studying ground states and Gibbs
measures on these systems. Also, for spin glasses the Branching-OGP provides a tight algorithmic characterization for known polynomial time algorithms.
It is remarkable that ultrametric (branching)-OGP found its way in models beyond spin glasses, 
where the ultrametricity does not arise in the study of ground states and Gibbs measure, as both can be estimated with far more elementary probabilistic methods such as the second moment method. We should also highlight that the analysis of the branching-OGP in this paper 
(as well as~\cite{du2025algorithmic})  is very elementary if not straightforward, and thus a side value of our paper is a simple demonstration of the power of the branching-OGP based methods.

While in the past OGP was used primarily as an obstruction to classes of algorithms exhibiting stability, its use for online (but otherwise not necessarily stable) algorithms is rather recent and was introduced first in~\cite{gamarnik2023geometric} in the context of the random perceptron model. 
In addition to~\cite{du2025algorithmic}, \cite{erba2025maximumaverage},\cite{abhishek2025large} and the present paper, it was also recently adopted in~\cite{gamarnik2025optimal} in the context of studying cliques in dense \ER graphs. It is anticipated that OGP will
emerge as a useful general technique to study obstructions to online algorithms in many other settings.

\subsection{Notation}
We use $k=k_n \to \infty$ for the size of the subtensor problem and will later describe assumptions on the rate of divergence w.r.t. $n$ for our results to hold.  For a tensor $G\in (\mathbb R^n)^{\otimes p}$ and subsets $I_1,\dots,I_p\subset [n]$, we denote $G_{I_1,\dots,I_p}$ as the subtensor:
\begin{equation}
    G_{I_1,\dots,I_p} = \{G_{i_1,\cdots,i_p}|i_s\in I_s\text{ for all }1\leq s\leq p\}\,.\notag
\end{equation}
We denote $\operatorname{Ave}(G_{I_1,\dots, I_p})$ and $\operatorname{Sum}(G_{I_1,\dots, I_p})$ as the average and the sum of entries of the subtensor $G_{I_1,\dots, I_p}$ respectively. We define
\begin{equation}
    O_k\triangleq\{(i_1,\dots,i_p):1\leq i_1,\dots,i_p\leq k\} = [k]^p\,.\notag
\end{equation}
For a tensor $G\in (\mathbb R^n)^{\otimes p}$, we define $G_{\leq k}$ to be the subtensor 
\begin{equation}\label{eq-subtensor-in-corner}
    G_{\leq k} \triangleq \{G_{i_1,\dots,i_p}:(i_1,\dots,i_p)\in O_k\}\,.
\end{equation}
For a rooted tree $\mathcal{T}$, we denote $\mathcal{L}$ as the set of leaves. For a vertex $u$, let $|u|$ be the depth of $u$ in the tree $\mathcal{T}$. For two vertices $u,v\in\mathcal{T}$, we define $u\wedge v$ to be the common ancestor of $u,v$ with the largest depth. For two sequences $\{a_n\}$ and $\{b_n\}$, we use standard Landau notation and for e.g. write $a_n=o_n(b_n)$ if $a_n/b_n\to 0$ as $n\to \infty$, and $a_n=O_n(b_n)$ if $a_n\leq Cb_n$ holds for some constant $C>0$ which does not depend on $n$. 

The next Section contains our main results. We provide a brief discussion and future directions in Section \ref{sec:disc}.  The remaining Sections contain proofs.

\section{Main results}
We start in Section \ref{sec:res-pos} by describing an explicit algorithm called
Incremental Greedy Procedure ($\mathcal{IGP}$), which searches for large average sub-tensors and describe its performance, namely asymptotics for the average of a typical output of this algorithm. This will largely follow \cite{GL18AOS}.
Then in Section \ref{sec:res-neg} we introduce the class of online algorithms, of which the algorithm $\mathcal{IGP}$ will be a special case. Finally we show our main result:
within the general class of online algorithms, the performance of $\mathcal{IGP}$  
cannot be improved upon.

\subsection{Positive side: online algorithms}
\label{sec:res-pos}
In this section, we start by introducing an algorithm for finding the densest subtensor, which is essentially a generalization of the $\mathcal{IGP}$ algorithm in the matrix ($p=2$) setting  $G = (G_{i,j})\in \mathbb{R}^{n\times n}$ in \cite{GL18AOS}. We start by briefly recalling the intuition for this algorithm from \cite{GL18AOS} in the matrix setting and then give a precise description of the algorithm in the general $p$ setting.  Start with an arbitrary row $i_1$, consider the row $G_{i_1, [n]}$ and find the largest element $G_{i_1,j_1}$. Next consider the column $G_{[n], j_1}$ and find the row $i_2\neq i_1$ with the largest $G_{i_2, j_1}$ value.  Now consider the $2\times n$ submatrix $G_{\{i_1, i_2\}, [n]}$ and find the column $j_2\neq j_1 $ such that the sum of the $2\times 1$ $G_{\{i_1, i_2\}, j_2}$ submatrix is as large as possible. Iterate this procedure alternating between adding rows and adding columns till we hit $k$ rows and $k$ columns. This is the essential \emph{raison d'\^etre} of $\mathcal{IGP}$ algorithm, however directly analyzing the asymptotic performance of this algorithm is non-trivial owing to the dependence created as one sequentially adds rows and columns. This issue is easily circumvented by partitioning the entire $n\times n$ matrix into $k\times k$ equal size groups and only searching for the ``best row'' or ``best column'' in this group thus preserving independence as the algorithm proceeds. 

Let us now describe the precise algorithm, now in the setting with general $p$. Given $n\in\mathbb Z^+$, we partition $[n]$ into $k+1$ disjoint sets, where the first $k$ subsets are defined as follows:
\begin{equation}\label{eq:P-in}
    P_{i,n}=\{(i-1)\lfloor n/k\rfloor+1,(i-1)\lfloor n/k\rfloor+2,\dots,i\lfloor n/k\rfloor\}\text{ for }1\leq i \leq k\,.
\end{equation}

\begin{breakablealgorithm}\label{alg: IGP algorithm}
\caption{Incremental Greedy Procedure ($\mathcal{IGP}$)}
\begin{algorithmic}[1]
\State \textbf{Input:} A tensor $G\in (\mathbb R^n)^{\otimes p}$ and an integer $k$.
\State \textbf{Initialize:} Select $i_1,\dots,i_{p-1}\in P_{1,n}$ arbitrarily and set $I_1=\{i_1\},...,I_{p-1}=\{i_{p-1}\}$. Find $i_p\in P_{1,n}$, such that $\operatorname{Ave}(G_{I_1,\dots,I_{p-1},i_p})\geq \operatorname{Ave}(G_{I_1,\dots,I_{p-1},j})$ for all $j\in P_{1,n}$. Let $I_p=\{i_p\}$.
\State \textbf{While} $|I_1|=|I_2|=\cdots=|I_p|<k$ \textbf{do}
    \State \indent \textbf{For} $1\leq s\leq p$ \textbf{do}
    \State \indent \indent Find the index $i_s\in P_{|I_s|+1,n}$ such that $$\operatorname{Ave}(G_{I_1,\dots,I_{s-1},i_s,I_{s+1},\dots,I_p})\geq \operatorname{Ave}(G_{I_1,\dots,I_{s-1},j,I_{s+1},\dots,I_p})\text{ for all }j\in P_{|I_s|+1,n}.$$
    \State \indent \indent  Set $I_s=I_s\cup\{i_s\}$.
    \State \indent \textbf{End For}
    \State \textbf{End While}
\State \textbf{Return} $G_{I_1,\dots,I_p}$.
\end{algorithmic}
\end{breakablealgorithm}

\begin{theorem}
\label{thm:pos}
     Let $k= \exp(o(\log n))$. Denote the output of Algorithm~\ref{alg: IGP algorithm} by $G_{\mathcal{IGP}}$. For any $\epsilon>0$, we have the following:
    \begin{equation}
        \mathbb P\left[ \operatorname{Ave}(G_{\mathcal{IGP}})\geq \left(\frac{2p}{1+p}-o_k(1)-\epsilon\right) \sqrt{\frac{2\log n}{k^{p-1}}}  \right] = 1-o(1)\,.\notag
    \end{equation}
\end{theorem}

\subsection{Computational hardness for online algorithms}
\label{sec:res-neg}

We now introduce the class of online algorithms that our analysis covers. 

\begin{definition}\label{def-online}
    Fix an algorithm $\mathcal{A}$ that takes a tensor $G\in (\mathbb R^n)^{\otimes p}$ as input and outputs a subtensor in $(\mathbb R^k)^{\otimes p}$ denoted $\mathcal{A}(G)$. 
    Recall 
        \eqref{eq-subtensor-in-corner} for the principle subtensor of a given tensor.
    The algorithm is called \emph{online} if the output $\mathcal{A}(G)$ is generated sequentially by constructing $\{\mathcal{A}(G)_{\leq r}\setminus \mathcal{A}(G)_{\leq r-1}:1\leq r\leq k\}$ according to the following rule:
    \begin{enumeratea}
        \item  
        $\mathcal{A}(G)_{1}$ is constructed based on $\{G_{i_1,\dots,i_p}:i_1,\dots,i_p\leq \lfloor\frac{1}{k}n\rfloor\}$.
    \item For each $1\leq s\leq k$, assuming 
    $\{\mathcal{A}(G)_{\leq r}\setminus \mathcal{A}(G)_{\leq r-1}
    :r\leq s-1\}$ are constructed,  $\{\mathcal{A}(G)_{\leq s}\setminus \mathcal{A}(G)_{\leq s-1}\}$ is determined by $\{G_{i_1,\dots,i_p}:i_1,\dots,i_p\leq \lfloor\frac{s}{k}n\rfloor\}$ and $\{\mathcal{A}(G)_{\leq r}\setminus \mathcal{A}(G)_{\leq r-1}
    :r\leq s-1\}$.
    \end{enumeratea}
\end{definition}
We denote the set of online algorithms by $\operatorname{OLA}$.
Note that the algorithm Increment Greedy Procedure ($\mathcal{IGP}$) is in $\operatorname{OLA}$. We now state our main result regarding tight upper bounds.
\begin{theorem}\label{thm-hardness-online}

Suppose $k=k_n\to\infty$.

For every $\epsilon>0$ and $p$ there
     exists $c=c(\epsilon,p)>0$ such that for any $\mathcal{A}\in\operatorname{OLA}$,
    \begin{equation}
        \mathbb P\left[ \operatorname{Ave}(\mathcal{A}(G))\geq \left( \frac{2p}{1+p}+\epsilon+o_k(1)\right)\sqrt{\frac{2\log n}{k^{p-1}}} \right]\leq \exp(-ck\log n)\,.\notag
    \end{equation}
\end{theorem}

\subsection{Discussion}
\label{sec:disc}
We now provide a brief discussion, placing the results in this paper in the context of recent results and describe open directions to pursue in the future. 
To display the main ingredients of the phenomenon, this paper focused largely on the Gaussian disorder regime. The general proof techniques should carry over in a straightforward fashion to other assumptions on the driving disorder of the underlying tensor when the entries are independent and identically distributed with sub-Gaussian tail behavior. Simulation driven evidence suggest different behavior for sub-exponential tails.  Next, it would be interesting to see how far such barriers to efficient algorithms carry over when the underlying entries are perhaps weakly dependent. One natural first step is understanding the setting of correlation matrices where the underlying matrix results from computing correlations between a large collection of features observed on a collection of individuals and the goal is extracting subsets of features pertaining to large  average correlations within them; such questions turn out to be of importance in areas ranging from network neuroscience to itemset mining in the context of preference mining in social settings, and have inspired a host of literature both on extremal behavior of such correlations under various assumptions on the driving data \cite{cai2011limiting,cai2011limiting} as well as iterative algorithms and their empirically observed limits in mining subsets with large average correlations \cites{bodwin2018testing,dewaskar2023finding,mosso2017latent}. It would be interesting to see how far the techniques in this paper can be pushed to understand natural limits in these situations.

\section{Positive side: algorithm in finding the densest subtensor}
\label{sec:positive}

{
The goal of this Section is to prove Theorem \ref{thm:pos}. We need some additional notation. 
For each $1\leq s\leq p$, we denote $I_s^{(r)}$ as the set $I_s$ constructed by step $r$ ($|I_s^{(r)}|=r$ and $1\leq r\leq k$). Define
\begin{equation}
    M_s^{(r)}\triangleq\max_{j \in P_{r,n}} M_s^{(r)}(j)\,,
\end{equation}
where $$M_s^{(r)}(j) = \sum_{i_1\in I_1^{(r)}}\sum_{i_2\in I_2^{(r)}}\dots\sum_{i_{s-1}\in I_{s-1}^{(r)}}\sum_{i_{s+1}\in I_{s+1}^{(r-1)}}\cdots\sum_{i_p\in I_p^{(r-1)}} G_{i_1,\dots,i_{s-1},j,i_{s+1},\dots,i_p}\,.$$
Recalling the definition of $P_{i,n} $ in \eqref{eq:P-in}, we have that the random variables
$$\{M_s^{(r)}(j):1\leq s \leq p,1\leq r\leq k,j\in P_{r,n}\}$$
are independent.
The following lemma provides a guarantee for the lower bound of $M_s^{(r)}$.

\begin{lemma}\label{lem:est-Msr}
For any $\epsilon>0$, we have
    \begin{equation}
        \mathbb P\left[\forall 1\leq r\leq k,1\leq s\leq p\,, M_s^{(r)} \geq \sqrt{r^{s-1}(r-1)^{p-s}}\cdot \sqrt{(2-\epsilon)\log n} \right] = 1-o(1)\,.\notag
    \end{equation}
    We denote this event by $\mathcal{G}$.
\end{lemma}
\begin{proof}
 For each $j\in P_{r,n}$, $M_s^{(r)}(j)$ obeys $\mathbf{N}(0,r^{s-1}(r-1)^{p-s})$. Therefore, by standard gaussian tail estimate, we have
    \begin{equation}
        \mathbb P\left[ M_s^{(r)}(j)\geq  \sqrt{r^{s-1}(r-1)^{p-s}}\cdot \sqrt{(2-\epsilon)\log n}\right] = \frac{1}{n^{1-\frac{\epsilon}{2}+o(1)}}\,.\notag
    \end{equation}
    Therefore, by independence we obtain
    \begin{equation}
    \begin{split}
        &\mathbb P\left[M_s^{(r)}\geq \sqrt{r^{s-1}(r-1)^{p-s}}\cdot \sqrt{(2-\epsilon)\log n}\right] = 1-\left(1-\frac{1}{n^{1-\frac{\epsilon}{2}+o(1)}}\right)^{n/k}\\
        =&1-\exp\left(-n^{\frac{\epsilon}{2}+o(1)}\right)\,.
    \end{split}\notag
    \end{equation}
    By applying a union bound on $1\leq s\leq p$ and $1\leq r\leq k$, we obtain the desired result.
\end{proof}
Then we present the proof of Theorem~\ref{thm:pos}.
\begin{proof}[Proof of Theorem~\ref{thm:pos}]
By Lemma~\ref{lem:est-Msr}, we have that $\mathcal{G}$ happens with probability $1-o(1)$. Now suppose that $\mathcal{G}$ holds, which implies
\begin{equation}
\begin{split}
    &\operatorname{Sum}(G_{\mathcal{IGP}})\geq \sum_{r=2}^k\sum_{s=1}^p M_s^{(r)}\geq (1-\epsilon)\sqrt{2\log( n/k)}\sum_{r=2}^k\sum_{s=1}^p r^{\frac{s-1}{2}}(r-1)^{\frac{p-s}{2}}\\
    \geq& (1-\epsilon)\sqrt{2\log( n/k)}\sum_{r=2}^k p (r-1)^{\frac{p-1}{2}}=(1-\epsilon)\sqrt{2\log( n/k)}pk^{\frac{p+1}{2}}\sum_{r=1}^{k-1} \frac{1}{k} \left(\frac{r}{k}\right)^{\frac{p-1}{2}}\\
    =&(1-\epsilon)pk^{\frac{p+1}{2}}\sqrt{2\log( n/k)}\times(1-o_k(1))\int_0^1 x^{\frac{p-1}{2}}dx\\
    =&(1-\epsilon-o_k(1))\frac{2p}{p+1}k^{\frac{p+1}{2}}\sqrt{2\log (n/k)}\,,
\end{split}\notag
\end{equation}
where the second inequality is from the definition of $\mathcal{G}$, in the third inequality we used $r^{\frac{s-1}{2}}\geq (r-1)^{\frac{s-1}{2}}$, and the second equality follows from the definition of Riemann integral. By taking an average of $G_{\mathcal{IGP}}$, the conclusion of Theorem~\ref{thm:pos} follows. 
\end{proof}
}

\section{Negative side: algorithm obstruction via ultrametric tree}
In this section, we establish the hardness result for online algorithms. 
The proof hinges on the branching-OGP structure, which was first introduced in \cite{huang2025tight} in the context of optimizing the mean-field spin glass model. Subsequently, \cite{du2025algorithmic} leveraged branching-OGP to characterize the computational hardness of the graph alignment problem. We start by describing some general constructions and technical bounds in Section \ref{sec:bogp}. These constructions and bounds are then used to prove Theorem \ref{thm-hardness-online} in Section \ref{sec:proof-neg}.

\subsection{The branching overlap gap property in densest subtensor}
\label{sec:bogp}
Before we formally define the branching OGP, we need some notation. Let $\epsilon>0$. We will use Riemann integral approximation to corresponding sums. Partition the interval $[0,1]$ using a sequence $0=\alpha_0<\alpha_1<\cdots<\alpha_N=1$ and write $\Delta:=\max_{1\leq i\leq n} (\alpha_i-\alpha_{i-1})$. Note that
\begin{equation}
    \begin{split}
        \sum_{i=1}^N(\alpha_i-\alpha_{i-1})(\alpha_{i-1})^{\frac{p-1}{2}}\leq \sum_{i=1}^N (\alpha_i-\alpha_{i-1})\sqrt{\frac{\sum_{j=0}^{p-1}(\alpha_{i-1})^j(\alpha_i)^{p-1-j}}{p}}{\leq} \sum_{i=1}^N(\alpha_i-\alpha_{i-1})(\alpha_{i})^{\frac{p-1}{2}}\,,
    \end{split}\notag
\end{equation}
where in the second inequality we use $\sum_{j=0}^{p-1}(\alpha_{i-1})^j(\alpha_i)^{p-1-j}\leq p\alpha_i^{p-1}$ which follows using $\alpha_{i-1}\leq\alpha_i$.
Thus, by the property of Riemann integral, for any sequence of partitions with $\Delta \to 0$ we have
\begin{equation}
\label{eqn:153}
    \lim_{\Delta\to 0} p\sum_{i=1}^N (\alpha_i-\alpha_{i-1})\sqrt{\frac{\sum_{j=0}^{p-1}(\alpha_{i-1})^j(\alpha_i)^{p-1-j}}{p}}=p\int_0^1x^{\frac{p-1}{2}}\operatorname{d}x=\frac{2p}{p+1}:=\kappa_p\,.
\end{equation}
Thus for any $\epsilon>0$, we can choose $N=N(\epsilon,p)$ and a collection $0=\alpha_0<\alpha_1<\cdots<\alpha_N=1$, such that with $\kappa_p$ as in \eqref{eqn:153},
\begin{equation}\label{eq-approx-integral-tensor}
     p\sum_{i=1}^N (\alpha_i-\alpha_{i-1})\sqrt{\frac{\sum_{j=0}^{p-1}(\alpha_{i-1})^j(\alpha_i)^{p-1-j}}{p}}\leq \kappa_p+\frac{\epsilon}{10}\,.
\end{equation}
For future use write, 
$$1+\delta:= \frac{\kappa_p+\epsilon}{\kappa_p+\frac{\epsilon}{10}}.$$ 
Fix a large constant $D$ satisfying 
\begin{equation}\label{eq-choice-D}
    D>\max_{1\leq \ell \leq N}\frac{\alpha_{\ell-1}}{2\delta(\alpha_\ell-\alpha_{\ell-1})}\,.
\end{equation}
\begin{definition}[Forbidden structure in $p$-tensor] For a $D$-regular tree $\mathcal{T}$ with depth $N$ (where the depth of the root is assumed $1$) and leave set $\mathcal{L}$. Consider a sequence $0=\alpha_0<\alpha_1<\dots<\alpha_N=1$ satisfying \eqref{eq-approx-integral-tensor}. For each vertex $v\in\mathcal{T}$ with depth $1\leq j\leq N$,
we assign for each  $1\leq s\leq p$ a subset $A^{(s)}_v\subset [n]$ with $|A^{(s)}_v|=(\alpha_j-\alpha_{j-1})k$. For any $1\leq s\leq p$, we require $A^{(s)}_v\cap A^{(s)}_{v'}=\emptyset$ for any distinct $v$ and $v'$. We further assign a set of tensor entry indices for this $v$, which is defined by
\begin{equation}
    E_v\triangleq\bigcup_{s=1}^{p} \left( \bigotimes_{r=1}^{s-1} \bigcup_{i=1}^jA^{(r)}_{v(i)} \otimes A^{(s)}_v\otimes \bigotimes_{r=s+1}^{p} \bigcup_{i=1}^jA^{(r)}_{v(i)} \right)\,,
\end{equation}
where $v(i)$ is the ancestor of $v$ with depth $i$. For the set of vertices in the leaf set, namely for each $v\in\mathcal{L}$, we define $M_v$ as the subtensor in $(\mathbb R^k)^{\otimes p}$ formed by the indices in $\cup_{i=1}^N E_{v (i)}$. 
Then we say a set of such tensors $\{M_v\}_{v\in\mathcal{L}}$ has the \emph{forbidden structure} if 
\begin{equation}\label{eq-def-fobidden}
    \operatorname{Sum}(M_v)\geq \left(\kappa_p+\epsilon\right)\underbrace{\sqrt{2k^{p+1}\log n}}_{\fD_n} \text{ for all }v\in\mathcal{L}\,.
\end{equation}
\end{definition}

Note that the above is equivalent to 
\[\avg(M_v) \geq \left(\kappa_p+\epsilon\right)\sqrt{\frac{2\log n}{k^{p-1}}}, \]
for all $v\in \mathcal{L}$. 

\begin{proposition}
    \label{prop:forbidden-structure-tensor}
    Denote $\mathfrak{F}$ as the event that such a forbidden structure exists. Then there exists a constant $c=c(p)>0$,  that depends on $p$, independent of $k$,  such that
    \begin{equation}\label{eq-prop-forbidden-structure-tensor}
        \mathbb P\left[\mathfrak{F}\right]\leq \exp\left(-cpk\log n\right)\,.
    \end{equation}
\end{proposition}
\begin{proof}
    For each $v\in\mathcal{T}$, we denote $\gamma_v^*$ as the random variable such that
    \begin{equation}
        \sum_{(i_1,\dots,i_p)\in E_v} G_{i_1,\dots,i_p}=\gamma_v^*\fD_n\,.
    \end{equation}
    For any $v\in\mathcal{L}$, summing the above equality along the ray connecting the root to $v$, we obtain
    \begin{equation}
        \sum_{\ell=1}^N \sum_{(i_1,\dots,i_p)\in E_{v(\ell)}}G_{i_1,\dots,i_p}=\operatorname{Sum}(M_v) = \sum_{\ell=1}^N \gamma_{v(\ell)}^* \fD_n\,.
    \end{equation}
    By \eqref{eq-def-fobidden} on the event $\fF$, we have
    \begin{equation}
        \sum_{\ell=1}^N \gamma^*_{v(\ell)}\geq \kappa_p+\epsilon \text{ for all }v\in\mathcal{L}\,.
    \end{equation}
    Summing over the above expression and then taking the average, we get
    \begin{equation}\label{eq-sum-and-average}
    \sum_{\ell=1}^N \frac{1}{D^{\ell-1}}\sum_{v:|v|=\ell}\gamma_v^*\geq \kappa_p+\epsilon\,.
    \end{equation}
    Using \eqref{eq-sum-and-average} and \eqref{eq-approx-integral-tensor}, we have that there exists an $\ell\in[N]$ such that
    \begin{equation}
        \begin{split}
            &\frac{1}{D^{\ell-1}}\sum_{v:|v|=\ell}\gamma_v^*\geq \frac{\kappa_p+\epsilon}{\kappa_p+\epsilon/10} (\alpha_\ell-\alpha_{\ell-1})\sqrt{p\sum_{j=0}^{p-1}(\alpha_{i-1})^j(\alpha_i)^{p-1-j}}\\
            =& (1+\delta)(\alpha_\ell-\alpha_{\ell-1})\sqrt{p\sum_{j=0}^{p-1}(\alpha_{i-1})^j(\alpha_i)^{p-1-j}}\,.
        \end{split}
    \end{equation}
    We denote the above event by $\mathcal{E}_\ell$, then we have
    \begin{equation}
        \mathbb P\left[ \mathfrak{F}\right]\leq \sum_{\ell=1}^N \mathbb P\left[\mathcal{E}_\ell\right]\,.
    \end{equation}
    It remains for us to give an upper bound for $\mathbb P\left[\mathcal{E}_\ell\right]$ for each $\ell$. Our strategy is to control $\mathbb P\left[\mathcal{E}_\ell\right]$ and then apply the union bound. To this end, the first step is to compute the enumerations of
    \begin{equation}\label{eq-choices-before-ell-tensor}
        \{A^{(s)}_v:1\leq s\leq p, |v|\leq \ell\}\,.
    \end{equation}
    which is bounded by (this follows from controlling the enumerations of each vertex $v$ with $|v|\leq \ell$)
    \begin{equation}\label{eq-enum-each-ell-tensor}
        \operatorname{Enum}_\ell:=n^{pk\sum_{i=1}^\ell D^{i-1}(\alpha_i-\alpha_{i-1})}=\exp\left( \left(pk\sum_{i=1}^\ell D^{i-1}(\alpha_i-\alpha_{i-1})\right)\log n \right)\,.
    \end{equation}
   
    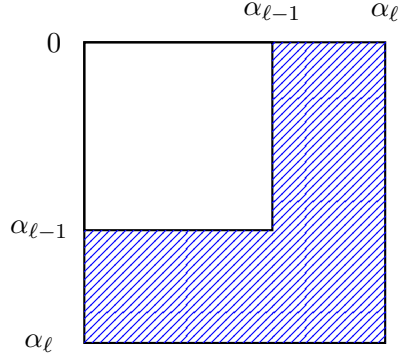
\begin{figure}
    \centering
   \begin{tikzpicture}[scale=1, yscale=-1]

  \def\a{4}  
  \def\b{2.5}

  \draw[thick] (0,0) rectangle (\a,\a);

  \draw[thick] (0,0) rectangle (\b,\b);

  \fill[pattern=north east lines, pattern color=blue] (\b,0) rectangle (\a,\a); 
  \fill[pattern=north east lines, pattern color=blue] (0,\b) rectangle (\b,\a);

  \node at (-0.4, 0) {$0$};
  \node at (-0.6, \b) {$\alpha_{\ell-1}$};
  \node at (-0.6, \a) {$\alpha_\ell$};

  \node at (\b, -0.4) {$\alpha_{\ell-1}$};
  \node at (\a, -0.4) {$\alpha_\ell$};

\end{tikzpicture}

\caption{In the setting $p=2$, showing density of points covered by $E_v$ }
    \label{fig:alpha-l}
    
\end{figure}
    
    Fixing the choices of \eqref{eq-choices-before-ell-tensor}, and note that the number of entries in each $E_v$ is $(\alpha_\ell^p-\alpha_{\ell-1}^p)k^p$, see Figure \ref{fig:alpha-l}. Further there are $D^{\ell-1}$ number of vertices in $\mathcal{T}$ with depth $\ell$.

 Thus we obtain,
    \begin{equation}
\sum_{v:|v|=\ell}\sum_{(i_1,\dots,i_p)\in E_v} G_{i_1,\dots,i_p}\sim\mathbf{N}\left(0,D^{\ell-1}(\alpha_\ell^p-\alpha_{\ell-1}^p)k^p\right)\,.
    \end{equation}
    This implies, 
    \begin{equation}\label{eq-est-prob-tensor}
    \begin{split}
        &\operatorname{Prob}_\ell  := \mathbb P\left[ \sum_{v:|v|=\ell}\gamma_v^*\geq D^{\ell-1}(1+\delta)(\alpha_\ell-\alpha_{\ell-1})\sqrt{p\sum_{j=0}^{p-1}(\alpha_{i-1})^j(\alpha_i)^{p-1-j}} \right]\\
        =&\mathbb P\left[  \sum_{v:|v|=\ell}\sum_{(i_1,\dots,i_p)\in E_v} G_{i_1,\dots,i_p} \geq D^{\ell-1}(1+\delta)(\alpha_\ell-\alpha_{\ell-1})\sqrt{p\sum_{j=0}^{p-1}(\alpha_{i-1})^j(\alpha_i)^{p-1-j}} \mathfrak D_n \right]\\
        \leq &\exp\left(-\frac{(D^{\ell-1})^2(1+\delta)^2(\alpha_\ell-\alpha_{\ell-1})^2p\sum_{j=0}^{p-1}(\alpha_{i-1})^j(\alpha_i)^{p-1-j} \mathfrak D_n^2}{2D^{\ell-1}(\alpha_\ell^p-\alpha_{\ell-1}^p)k^p}\right)\\
        = & \exp\left( -p k D^{\ell-1}(1+\delta)^2 (\alpha_\ell-\alpha_{\ell-1})\log n \right)\,.
    \end{split}
    \end{equation}
    Combining \eqref{eq-enum-each-ell-tensor} and \eqref{eq-est-prob-tensor} yields
    \begin{equation}
        \begin{split}
            &\mathbb P\left[\mathcal{E}_\ell\right]\leq \operatorname{Enum}_\ell\times \operatorname{Prob}_\ell\\
            \leq & \exp\left( pk\log n\left( \sum_{i=1}^\ell D^{i-1}(\alpha_i-\alpha_{i-1}) - D^{\ell-1}(1+\delta)^2(\alpha_\ell-\alpha_{\ell-1}) \right) \right)\\
            = & {\exp\left( pk\log n\left( \sum_{i=1}^{\ell-1}D^{i-1}(\alpha_i-\alpha_{i-1}) +(1-(1+\delta)^2)D^{\ell-1}(\alpha_\ell-\alpha_{\ell-1}) \right) \right)}\\
        \leq & {\exp\left( pk\log n\left( \sum_{i=1}^{\ell-1}D^{\ell-2}(\alpha_i-\alpha_{i-1}) +(-2\delta-\delta^2)D^{\ell-1}(\alpha_\ell-\alpha_{\ell-1}) \right) \right)}\\
        \leq  &\exp\left( pk\log n\left( D^{\ell-2}\alpha_{\ell-1} - D^{\ell-1}2\delta (\alpha_\ell-\alpha_{\ell-1}) \right) \right)\,,
        \end{split}
    \end{equation}
    where the second-to-last line follows from the inequality $D^{i-1}\leq D^{\ell-2}$ for all $i\leq \ell-1$ and in the last line we apply the telescoping sum. Recalling the definition of $D$ in \eqref{eq-choice-D}, we have that the above expression is bounded by $\exp(-cpk\log n)$. The proof is complete. 
\end{proof}

\subsection{Completing the proof of Theorem \ref{thm-hardness-online}:}
\label{sec:proof-neg}
 In the following, we introduce the notion of \emph{correlated instance} which serves as an ingredient in the hardness result. Consider the tree $\mathcal{T}$ and the sequence $\vec{\alpha}=(\alpha_0,\alpha_1,\dots,\alpha_N)$ defined in Section \ref{sec:bogp}.

\begin{definition}[$(\mathcal{T},\vec{\alpha})$-\emph{correlated instance}]
    For each vertex $v\in\mathcal{T}$, we assign a tensor $E^v$ with independent standard Gaussian random entries $\{E^v_{(i_1,\dots,i_p)}:i_1,\dots, i_p\leq n\}$. A $(\mathcal{T},\vec{\alpha})$-\emph{correlated instance} is a set of tensors $\{G^v\}_{v\in\mathcal{L}}$ indexed by $\mathcal{L}$ in $ (\mathbb R^n)^{\otimes p}$. For each $v\in \mathcal{L}$, 
    \begin{equation}
        G^v_{i_1,\dots,i_p}=E^{v(j)}_{i_1,\dots,i_p},\text{ for }(i_1,\dots,i_p) \in O_{\alpha_jn}\setminus O_{\alpha_{j-1}n}\,,j=1,...,N\,.
    \end{equation}
\end{definition}
For a tensor $G\in (\mathbb R^n)^{\otimes p}$, we denote $\operatorname{Suc}_{\mathcal{A}}(G)$ the event that $\mathcal{A}$ outputs a subtensor $\mathcal{A}(G)\in (\mathbb R^k)^{\otimes p}$ by taking $G$ as input such that $$\operatorname{Sum}(\mathcal{A}(G))\geq \left(\frac{2p}{p+1}+\epsilon\right)\fD_n\,.$$ Then we have the following lemma. 

\begin{lemma}
    \label{lem-suc-prob-tree}
    Suppose that for $G$ with i.i.d $\mathbf{N}(0,1)$ entries,
    \begin{equation}
        \mathbb P\left[ \operatorname{Suc}_\mathcal{A}(G) \right]= p_{\operatorname{suc}}\,.\notag
    \end{equation}
    Then for a $(\mathcal{T},\vec{\alpha})$-\emph{correlated instance} $\{G^v\}_{v\in\mathcal{L}}$, we have
    \begin{equation}\label{eq-ultimate-induction-on-tree}
        \mathbb P\left[ \operatorname{Suc}_\mathcal{A}(G^v)\text{ for all }v\in\mathcal{L}\right]\geq p_{\operatorname{suc}}^{|\mathcal{L}|}\,.
    \end{equation}
\end{lemma}
\begin{proof}
    Define $\mathcal{F}_i$ to be the sigma field generated by $\{G^v_{\leq \alpha_i n}: v\in\mathcal{L}\}$. For each $v\in \mathcal{T}$, let $\mathcal{L}(v)$ be the descendant of $v$ in $\mathcal{L}$. Moreover, we denote $\mathcal{S}(v)$ as the event that for each $u\in\mathcal{L}(v)$, $\operatorname{Sum}(\mathcal{A}(G^u))\geq (2p/(p+1)+\epsilon)\fD_n$. Let $v_o$ be the root of $\mathcal{T}$, then we note that the event in the LHS of \eqref{eq-ultimate-induction-on-tree} is $\mathcal{S}(v_o)$. In the following, we will show
    \begin{equation}
        \label{eq-induction-statement}
        \mathbb P\left[\mathcal{S}(v_o)\right]\geq \prod_{v:|v|=i}\mathbb P\left[\mathcal{S}(v)\right]\,,\forall~~ 0\leq i\leq N\,
    \end{equation}
    via an induction procedure. The base step of the induction $i=0$ trivially holds. Suppose that \eqref{eq-induction-statement} is true for $\ell$. For each $v$ with $|v|=\ell$, we denote its offspring by $v_1,\dots,v_D$, then we have $\mathcal{S}(v)=\cap_{j=1}^D\mathcal{S}(v_j)$. Conditioned on $\mathcal{F}_\ell$, the events $\mathcal{S}(v_j),1\leq j\leq D$ are independent. Then we obtain,    \begin{equation}
        \begin{split}
            \mathbb P\left[\mathcal{S}(v)\right]&=\mathbb E\left[ \prod_{j=1}^D\mathbb P\left[\mathcal{S}(v_j)|\mathcal{F}_\ell\right]\right]=\mathbb E\left[\left(\mathbb P\left[\mathcal{S}(v_1)|\mathcal{F}_\ell\right]\right) ^D\right]\\
            &\geq \left( \mathbb E\left[ \mathbb P\left[\mathcal{S}(v_1)|\mathcal{F}_\ell\right]  \right]  \right)^D=\left(\mathbb P\left[\mathcal{S}(v_1)\right]\right)^D =\prod_{j=1}^D\mathbb P\left[\mathcal{S}(v_j)\right]\,,
        \end{split}\notag
    \end{equation}
    where in the first equality we used the conditional independence; the second and the last equalities follow from the symmetries among $v_1,\dots,v_D$; and the inequality comes from applying Jensen's inequality. Then we get that \eqref{eq-induction-statement} holds for $\ell+1$. The induction procedure is complete, establishing \eqref{eq-induction-statement} for $0\leq i\leq N$. By taking $i=N$ in \eqref{eq-induction-statement}, we have
    \begin{equation}
        \mathbb P\left[ \mathcal{S}(v_o)\right]\geq \prod_{v\in\mathcal{L}}\mathbb P\left[\mathcal{S}(v)\right]=p_{\operatorname{suc}}^{|\mathcal{L}|}\,.\notag
    \end{equation}
    The proof is complete.
\end{proof}
\begin{proof}[Proof of Theorem~\ref{thm-hardness-online}]
    For $(\mathcal{T},\vec\alpha)$-correlated instance $\{G^v\}_{v\in\mathcal{L}}$, note that if $\operatorname{Suc}_\mathcal{A}(G^v)$ holds for all $v\in\mathcal{L}$, then $\{\mathcal{A}(G^v)\}_{v\in\mathcal{L}}$ forms \emph{forbidden structure} according to Definition~\ref{def-online}. Then we have
    \begin{equation}
        p_{\operatorname{suc}}^{|\mathcal{L}|}\leq \exp(-ck\log n)\,.
    \end{equation}
    This implies $p_{\operatorname{suc}}\leq \exp(-\frac{c}{D^N}k\log n)$, completing the proof.
\end{proof}

\section*{Acknowledgments}
This material is based upon work supported by the National Science Foundation under Grant
No. DMS-1928930, while all three authors were in residence at the Simons Laufer
Mathematical Sciences Institute in Berkeley, California, during the Spring 2025 semester. SB was partially supported by NSF DMS-2113662, DMS-2413928, and DMS-2434559 and RTG grant DMS-2134107. DG was partially supported by NSF Grant CISE 2233897. SG is partially supported by National Key R\&D program of China
(No. 2023YFA1010103) and NSFC Key Program (Project No. 12231002).  We thank Vittorio Erba, Shuangping Li and  Lenka Zdeborov\'a for enlightening conversations at the start of this work.

\bibliographystyle{plain} 
\bibliography{References}

\end{document}